\documentclass[final,usenames,dvipsnames]{siamart1116}
\makeatletter
\def\input@path{{../}{./}}
\makeatother

\usepackage{amsfonts,amsmath,amsbsy,amssymb,bm}

\newtheorem{remark}[theorem]{Remark}

\usepackage{relsize,exscale}
\usepackage{graphicx}
\graphicspath{{./figs/}}

\usepackage[lofdepth,lotdepth]{subfig}
\usepackage{float}
\usepackage{enumitem}

\usepackage{xcolor}  

\usepackage{hyperref}
\hypersetup{
	plainpages=false,
    colorlinks,
    linktocpage=true,   
    linkcolor={red!50!black},
    citecolor={blue!50!black},
    urlcolor={blue!80!black}
} 
\usepackage[hyperpageref]{backref}

\numberwithin{equation}{section}

\newcommand{\RR}{\mathbb{R}}
\newcommand{\ZZ}{\mathbb{R}}

\newcommand{\cF}{{\mathcal{F}}}
\newcommand{\bI}{{\mathbf{I}}}
\newcommand{\bA}{{\mathbf{A}}}
\newcommand{\bB}{{\mathbf{B}}}
\newcommand{\bM}{{\mathbf{M}}}

\newcommand{\cN}{{\mathcal{N}}}
\newcommand{\cJ}{{\mathcal{J}}}
\newcommand{\cS}{{\mathcal{S}}}
\newcommand{\cT}{{\mathcal{T}}}
\newcommand{\cR}{{\mathcal{R}}}

\newcommand{\bu}{\mathbf{u}}
\newcommand{\bv}{\mathbf{v}}
\newcommand{\vn}{\bm{n}}
\newcommand{\vx}{\bm{x}}
\newcommand{\vz}{\bm{z}}

\newcommand{\vH}{\bm{H}}
\newcommand{\vL}{\bm{L}}

\newcommand{\bsig}{\bm{\sigma}}
\newcommand{\btau}{\bm{\tau}}

\newcommand{\Fz}{\cF_{\vz}}
\newcommand{\subcT}{\raisebox{-0.3pt}{$\scriptscriptstyle{\cT}$}}
\newcommand{\bsigtz}{\bsig^{\Delta}_{\vz,\subcT}}

\newcommand{\divv}{\nabla\!\cdot\!}
\newcommand{\mcup}{\mathsmaller{\,\bigcup}}
\newcommand{\p}{\partial}

\newcommand{\at}[1]{\big\vert_{\raisebox{-0.5pt}{\scriptsize$#1$}} }
\newcommand{\abs}[1]{\left\vert {#1}\right\vert}
\newcommand{\norm}[1]{\left\Vert#1\right\Vert}
\newcommand{\uiter}{\bar{u}_{_\cT}}
\newcommand{\binprod}[2]{\bigl( {#1},\,{#2} \bigr)}
\newcommand{\diam}{\operatorname{diam\,}}

\newcommand{\jump}[2]{\lbrack\hspace{-1.5pt}\lbrack {#1} 
{\rbrack\hspace{-1.5pt}\rbrack}_{\raisebox{-2pt}{\scriptsize$#2$}} }

\newcommand{\RT}{\bm{\cR\!\cT}}
\newcommand{\vHdiv}{\vH(\operatorname{div};\Omega)}
\newcommand{\vhdiv}{\vH(\operatorname{div})}
\newcommand{\SigmaTz}{\bm{\Sigma}_{\vz,\subcT}}
\newcommand{\bsigtk}{\bsig^{\Delta}_{K,\subcT}}
\newcommand{\bsigt}{\bsig^{\Delta}_{\subcT}}
\newcommand{\bsigc}{\bsig^{c}_{\subcT}}

\newcommand{\sub}[1]{\raisebox{-2pt}{\scriptsize$#1$}}

\title{A Discretization-Accurate Stopping Criterion for Iterative Solvers for Finite Element Approximation\thanks{
This work was performed under the auspices of the U.S. Department of Energy by
Lawrence Livermore National Laboratory under Contract DE-AC52-07NA27344
({LLNL-JRNL-789117}).
This work was supported in part by the National Science Foundation
under grants DMS-1320608, DMS-1418934, and DMS-1522707.}}

\author{Zhiqiang Cai\thanks{
Department of Mathematics, Purdue University, 150 N. University
Street, West Lafayette, IN 47907-2067, zcai@math.purdue.edu.}
\and Shuhao Cao\thanks{
Department of Mathematics, University of California Irvine, Irvine, CA 92697, 
scao@math.uci.edu.}
\and Robert D. Falgout\thanks{Center for Applied Scientific Computing, 
Lawrence Livermore National Laboratory, 
Livermore, CA 94551-0808, falgout2@llnl.gov.}}

\date{}

\begin{document}
\maketitle

\begin{abstract}
This paper introduces a discretization-accurate stopping criterion of symmetric iterative methods for solving systems of algebraic equations resulting
from the finite element approximation. The stopping criterion consists of the 
evaluations of the discretization and the algebraic error estimators, that are based 
on the respective duality error estimator and the difference of two consecutive 
iterates. Iterations are terminated 
when the algebraic estimator is of the same 
magnitude as the discretization estimator. Numerical results for multigrid 
$V(1,1)$-cycle and symmetric Gauss-Seidel iterative methods are presented
for the linear finite element approximation to the Poisson equations. A large 
reduction in computational cost is observed compared to the 
standard residual-based stopping criterion.
\end{abstract}

\section{Introduction}
Consider the Dirichlet boundary value problem in a bounded polygonal/polyhedral 
domain $\Omega\subset \RR^d$ ($d=2,3$) for the diffusion equation as 
follows:
\begin{equation}
\label{eq:pb-ps}
\left\{
\begin{aligned}
-\divv(A \nabla u) &= f,  
&\; \text{ in }\, \Omega,
\\
u  &=g ,  &\; \text{ on }\, \p\Omega,
\end{aligned} 
\right.
\end{equation}
where $A$ is a scalar diffusion coefficient, and the data $f\in L^2(\Omega)$ and 
$g\in 
L^2(\partial \Omega)$.

In practice, the system of algebraic equations resulting from the finite element 
approximation to \eqref{eq:pb-ps} is often solved by iterative methods, e.g., 
Gauss-Seidel, conjugate gradient,
multigrid methods, etc. Instead of having 
the exact solution $u_{_\cT}$ of the algebraic system at hand, 
$\bar{u}_{_\cT} := {u}_{_\cT}^{(k)}$ is the current output from an iterative solver,
where $k$ is the number of iterations. The total energy error of $\bar{u}_{_\cT}$ to the solution $u$ of the 
continuous problem in \eqref{eq:pb-ps} consists of both discretization and algebraic errors as follows:
\begin{equation}
\label{eq:est-id}
\underbrace{\norm{u - \bar{u}_{_\cT}}_A^2}_\text{total error}
=\underbrace{\norm{u_{_\cT} - \bar{u}_{_\cT}}_A^2}_\text{algebraic error} + 
\underbrace{\norm{u - u_{_\cT}}_A^2}_\text{discretization error},
\end{equation}
where $\norm{\cdot}_A$ is the energy norm associated with the problem in
\eqref{eq:pb-ps} (for the norm notations, see
section~\ref{sec:prelim}).

The goal of this paper is to propose a stopping criterion for 
iterative solvers. To do so, we need to develop two error estimators for the 
respective discretization and algebraic errors. 
Since the discretization error is fixed for a given finite element space,
(\ref{eq:est-id}) clearly indicates that the stopping criterion of the iterative solver is when the algebraic estimator is of the same 
magnitude as the discretization estimator, provided that both represent their error 
counterparts reliably.

Discretization error estimators for the exact finite element approximations have been intensively studied during the past four decades
(see books \cite{Ainsworth-Oden,Verfurth-13} and references therein).
In the context of stopping criterion for iterative solvers, the residual-based a posteriori error estimator
was employed for the conforming finite element approximation by several researchers (see, e.g., 
\cite{Becker1995adaptive,Rannacher1999,Arioli2012interplay,Arioli2013stopping,Arioli2005stopping});
recovery-based estimators were used by Vohral\'{i}k, et al. in 
\cite{Jiranek2010} for the 
finite volume discretization and in \cite{Dolejvsi2013} for the discontinuous finite 
element approximation on non-matching grids.

In this paper, we will adopt the equilibrated flux error estimator 
(see, e.g.,  
\cite{Braess06eqres,Destuynder99explicit,Verfurth09note,Cai12eqrobust}) for the 
discretization error.
This is because the reliability bound of estimators of this type is constant free.
Using this technique, a locally post-processed flux based on the iterate $\uiter$ will be constructed. 
Unlike the exact finite element approximation $u_{_\cT}$, the local problems based on the current iterate $\uiter$ on vertex patches
are not consistent. To overcome this difficulty, we modify the local problems by adding back the algebraic errors. The resulting discretization
error estimator plus the algebraic error is proved to be reliable and the reliability bound for the discretization estimator 
component is constant free (Theorem \ref{thm:r}).

To construct the algebraic error estimator, we first bound $\big\Vert{u_{_\cT} - 
{u}_{_\cT}^{(k)}}\big\Vert_A$ above by
the energy norm of the difference of consecutive iterates, and the constant in the upper bound depends on the spectral radius of 
the error propagation operator (Theorem \ref{thm:ra}). The unknown spectral 
radius is further approximated by the ratio
of the $\ell^2$-norms of the residuals of consecutive iterates. The resulting 
algebraic error estimator is then proved 
to be reliable when sufficiently many iterations have been performed. 

Lastly, in Section \ref{sec:numex}, based on the discretizaton and algebraic 
estimators, 
a new stopping criterion for some given linear solvers is verified to reduce the 
computational cost numerically compared with its conventional residual-based 
counterpart used in \emph{hypre} \cite{Falgout2002hypre}. In addition, the numerics 
shows 
promising results in that the bounds are 
independent of the coefficient jump ratio even without the quasi-monotonicy 
assumption \cite{Bernardi00nonsmooth, Dryja96quasimonotone, Petzoldt02discont} for 
the distribution of the diffusion coefficient $A$.
\section{Finite element method and iterative solver}
\label{sec:prelim}
In this section, all preliminaries are presented. Denote $H^1(\Omega)$ 
with a specified boundary value as
$H^1_g(\Omega):= \{v\in H^1(\Omega):\, v= g  \; \text{ on }\, \p\Omega\}$,
and then the variational problem of \eqref{eq:pb-ps} is
\begin{equation}
\label{eq:pb-pw}
\text{Find } \; u \in H^1_g(\Omega)\; 
\text{ such that }\;
 \binprod{A \nabla u }{\nabla v} = \binprod{f}{v} , \quad 
\forall v\in H^1_{0}(\Omega), 
\end{equation}
where $\binprod{\cdot}{\cdot}$ denotes the $L^2$-inner product on the whole 
domain. 

Let $\cT = \{K\}$ be a triangulation of $\Omega$ using simplicial 
element, where $\cT$ is assumed to be quasi-uniform and regular. For each $K\in 
\cT$, $h_K := \diam (K) = O(|K|^{1/d})$.
The set of all the vertices of this triangulation is denoted by $\cN$. 
Throughout this paper, the term ``face'' is used to refer to the 
$(d-1)$-facet of a $d$-simplex in this triangulation ($d=2,3$). For the $d=2$ case, 
a face actually represents an edge. The set of all the interior faces is denoted by 
$\cF$. For any $F\in \cF$, $h_F:= \diam (F) =O(|F|^{1/(d-1)})$. 
Each face $F\in \cF$ is associated with a fixed unit normal $\vn_F$ globally. For 
any function or distribution $v$ well-defined on the two elements sharing a face $F$ 
respectively, define $\jump{v}{F} = v^- - v^+$ on an interior face. The $v^-$ and 
$v^+$ are defined in the limiting sense of $v^{\pm} = \lim\limits_{\epsilon\to 
0^{\pm}}v(\vx+\epsilon\vn_F)$. If $F$ is a boundary face, 
the function $v$ is extended by zero outside the domain to compute $\jump{v}{F}$. 
For every geometrical object $D$ and for every 
integer $k \geq 0$, $P_k(D)$ denotes the set of polynomials of degree 
$\leq k$ on $D$.

For the purpose of constructing the local error estimation procedure for the finite 
element approximation, notations of the following local geometric objects are used 
in this paper. First, denote by $\cN_K$ the set of all the vertices of $K\in \cT$. For any 
vertex $\vz \in \cN$, denote by
\[
\omega_{\vz} :=  \mcup_{ \{K\in \cT: \;\vz\in \cN_K \} } K
\]
as the vertex patch, which is the union of all elements sharing $\vz$ as a 
common vertex. Now $\cT_{\vz}$ stands for the triangulation of this patch such that 
$\cT_{\vz}:= \{K: K\subset\omega_{\vz}\}$. Denote
\[
\omega_K := \mcup_{\vz\in \cN_K}\omega_{\vz}
\] 
as the element patch for $K$ that contains all the elements sharing a vertex with 
$K$. For a face $F\in \cF$, denote the face patch as
\[
\omega_F := \mcup_{F\cap \p K\neq \emptyset} K,
\]
which contains the elements sharing $F$ as a common face. The $L^2$-inner 
product and norm on $\omega = \cup K\subset \Omega$ are denoted by 
\[
\binprod{u}{v}_{\omega} := \sum_{K\subset \omega} (u,v)_K \quad \mbox{and}\quad
\norm{v}_{0,\omega}^2 :=  ({v},{v})_{\omega},
\]
respectively. These notations carry through for vector-valued functions. 
The ``energy'' seminorm associated 
with the problem \eqref{eq:pb-pw} is (with slight abuse of notation, because the local 
seminorm is denoted as a norm):
\begin{equation}
\norm{v}_A^2 := \binprod{A \nabla u }{\nabla v} \;\text{ and }\; 
\norm{v}^2_{A,\omega} 
:= 
\binprod{A \nabla u }{\nabla v}_{\omega} .
\end{equation}

Let $\cF_K$ be 
the set of faces of an element $K\in\cT$. 
Denote the set of the interior faces within $\omega_{\vz}$ as:
\[
\Fz := \{F\in \cF: F\in \cF_K 
\text{ for }K\subset \omega_{\vz},\; F\cap \p\omega_{\vz} = \emptyset\}.
\]
Denote the $H^1$-conforming linear finite element space by
\begin{equation}
\label{eq:sp-p1}
\cS^1 := \{v\in H^1(\Omega): \, 
v\at{K}\in P_1(K),\; \,\,\forall \,\, K\in \cT \},
\end{equation}
and the piecewise constant space with respect to the triangulation $\cT$ by
\begin{equation}
\label{eq:sp-p0}
\cS^0 := \{v\in L^2(\Omega): \, 
v\at{K}\in P_0(K),\; \,\,\forall \,\, K\in \cT \},
\end{equation}
then the finite element approximation to 
\eqref{eq:pb-pw} is 
\begin{equation}
\label{eq:pb-pd}
\left\{
\begin{aligned}
&\text{Find } \; u_{_\cT} \in 
\cS^1\cap H^1_g(\Omega)\; \text{ 
such that}
\\[1mm]
& \binprod{A \nabla u_{_\cT} }{\nabla v}  = \binprod{f}{v}, \quad 
\forall \,\, v \in \cS^1\cap H^1_0(\Omega).
\end{aligned}
\right.
\end{equation}
For the presentation purpose, here it is assumed that both the diffusion coefficient 
$A$ and the data $f$ are in $\cS^0$, and denote $A\at{K} = A_K$, $f\at{K} = f_K$. 
Additionally, the Dirichlet boundary data $g$ can be represented by the trace of a 
function in $\cS^1$. In this setting, no data oscillation term will be present in 
the final error estimate bounds.

Let $\phi_{\vz_i}$ be the 
Lagrange nodal basis function of $\cS^1$ associated with an interior vertex 
$\vz_i\in \cN$.  Using these nodal basis functions, the discrete problem in \eqref{eq:pb-pd}
may be written as the following system of linear equations:
\begin{equation}
\label{eq:pb-pl}
\bA \bu =\mathbf{f},
\end{equation}
where the stiffness matrix $\bA$ 
is $ \bA[i,j] = a_{ij}$ with 
$a_{ij} = \binprod{A \nabla \phi_{\vz_j} }{\nabla  \phi_{\vz_i}}$; the $\bu$ is the 
vector representation of the exact solution $u_{_\cT}$; and        
the $\mathbf{f}$ is the vector representation of the right hand side 
with $i$-th row $\mathbf{f}[i]$ of $\mathbf{f}$ being $(f,\phi_{\vz_i})$. 
For a given initial guess $\bu^{(0)}$, an iterative solver for problem \eqref{eq:pb-pl}
has the following form
\begin{equation}
\label{eq:pb-it}
\bu^{(k+1)} = \bu^{(k)} + \bB(\mathbf{f} - \bA \bu^{(k)} ),
\end{equation}
where $\bu^{(l)}$ is the vector representation of the $l$-th iterate $u_{_\cT}^{(l)}$
for $l=0,\,1,\, \cdots$. Our attention in this paper is restricted to symmetric iterative 
methods, i.e., the matrix $\bB$ in \eqref{eq:pb-it} is symmetric. 

Next we define the norms for vectors and matrices: with the help of the context, the 
usual $2$-norm $\norm{\cdot}_2$ for a vector $\bv \in \RR^n$ and a non-singular 
symmetric matrix $\bM \in \RR^{n\times n}$ is 
defined by:
\begin{equation}
\label{eq:nm-m}
\norm{\bv}_2 := \sqrt{\bv\cdot \bv}  \text{ and }
\norm{\bM}_2 := \sup_{\norm{\bv}_2 = 1} \norm{\bM \bv}_2 = \rho(\bM),
\end{equation}
respectively,
where $\rho(\bM)$ is the spectral radius of $\bM$ equaling its largest eigenvalue.

The stiffness 
matrix $\bA$ is symmetric positive definite for the Dirichlet boundary value 
problem. As a result, $\bA^{1/2}$ is non-singular and can be used to induce a norm:
\begin{equation}
\norm{\bv}_{\bA} := \sqrt{\bA\bv\cdot \bv} = \norm{\bA^{1/2}\bv}_2
\text{ and }
\norm{\bM}_{\bA} := \sup_{\norm{\bv}_{\bA} = 1} \norm{\bM \bv}_{\bA}.
\end{equation}
By definition it is straightforward to verify that:
\begin{equation}
\label{eq:nm-ma}
\norm{\bM}_{\bA} = \sup_{\norm{\bA^{1/2}\bv}_2 = 1} \norm{\bA^{1/2} \bM \bv}_{2}
= \norm{\bA^{1/2} \bM \bA^{-1/2}}_{2}.
\end{equation}
For a finite element function $v$ and its vector representation $\bv$, the following 
equivalence between vector norm and Sobolev norm holds as well:
\begin{equation}
\norm{v}_{A} = \norm{\bv}_{\bA}.
\end{equation}

\section{Discretization error estimator using an equilibrated flux}
In this section, firstly the duality theory for the error estimation is introduced. 
Then a locally post-processed flux based on the iterate $\uiter:= u_{_\cT}^{(k)} $ 
for a fixed $k\geq 1$ is constructed. Lastly the reliability of the 
estimator based on this recovered flux is proved in order that a stopping criterion 
can be designed for the iterative solver. 

\subsection{Duality theory}
It is known that the variational problem in \eqref{eq:pb-pw} can be rewritten as a 
functional minimization problem, where the primal functional is:
\begin{equation}
\label{eq:fct-p}
\cJ(v):= \frac{1}{2}\,\binprod{A \nabla v }{\nabla v}
- \binprod{f}{v}
\end{equation}
Then problem (\ref{eq:pb-pw}) is equivalent to the following minimization problem:
\begin{equation}
\label{eq:pb-p}
\text{Find }u \in H^1_g(\Omega) \; \text{ such that } \;
\cJ(u) = \min_{v\in H^1_g(\Omega)} \cJ(v).
\end{equation}
 
The dual functional with respect to \eqref{eq:fct-p} is:
\begin{equation}
\label{eq:fct-d}
\cJ^{*}(\btau) := -\frac{1}{2}\binprod{A^{-1}\btau}{\btau}.
\end{equation}
The dual problem is then to maximize $\cJ^{*}(\btau)$ in the following 
space:
\begin{equation}
\label{eq:space-duality}
\bm{\Sigma} := \{ \btau \in \vHdiv:\, \divv \btau = f\},
\end{equation}
and can be phrased as:
\begin{equation}
\label{eq:pb-d}
\text{Find }\bsig \in \bm{\Sigma} \; \text{ such that } \;
\cJ^*(\bsig) = \max_{\btau\in \bm{\Sigma}} \cJ^*(\btau).
\end{equation}

The foundation to use the dual problem in constructing a 
posteriori error estimator is that the minimum of the primal functional $\cJ(\cdot)$ 
coincides with the maximum of the dual functional $\cJ^*(\bsig)$ (see 
\cite{Ekeland-Temam} Chapter 3):
\begin{equation}
\label{eq:mm}
\cJ(u) = \cJ^*(\bsig) \; \text{ and } \;\bsig = - A\nabla u .
\end{equation}
Now that \eqref{eq:mm} is satisfied, then a 
guaranteed upper bound can be obtained as follows: for any 
$\bsig_{_\cT} \in \bm{\Sigma}_{\subcT}:= \bm{\Sigma}\cap \RT^0 $ being a subspace of 
$\bm{\Sigma}$, where $\RT^0$ is 
the lowest order Raviart-Thomas element (e.g., see \cite{Brezzi-Fortin}),
\begin{equation}
\label{eq:est-u}
 \norm{u-\uiter}_A^2 =2\Big(\cJ(\uiter) - \cJ(u) \Big)
= \;2 \Big(\cJ(\uiter) - \cJ^*(\bsig) \Big)
\leq  2\Big( \cJ(\uiter) - \cJ^*(\bsig_{_\cT}) \Big).
\end{equation}
One of the main goals of this paper is to locally 
construct such $\bsig_{_\cT}$ based on the current iterate
$\uiter$, so that the global reliability bound in \eqref{eq:est-u} is automatically 
met.

\subsection{Localized flux recovery}
\label{sec:fr}
Let $\bsig^{\Delta}$ be the correction from the numerical flux 
$\overline{\bsig}_{_\cT}:=-A\nabla \uiter$ to the true 
flux $\bsig:= -A\nabla u$: 
\begin{equation}
\bsig^{\Delta} := \bsig - \overline{\bsig}_{_\cT}  
\end{equation}
Decompose $\bsig^{\Delta}$ by a partition of unity 
$\{\phi_{\vz}\}_{\vz\in \cN}$, which is the set of the nodal basis functions for 
the linear finite element space $\cS^1$, as follows:
\begin{equation}
\bsig^{\Delta} = \sum_{\vz\in \cN} \bsig^{\Delta}_{\vz} \;\text{ with } \;
\bsig^{\Delta}_{\vz} := \phi_{\vz} \bsig^{\Delta}.
\end{equation}
Denote the element residual on an element $K$ and 
the jump of the normal component of the numerical flux on a face $F$ by
\begin{eqnarray}
&& {r}_K:= \bigl\{f+ \divv(A\nabla \uiter)\bigr\}\at{K}  = f_K 
\\[2mm]
\mbox{and } \quad &&
\label{eq:jump}
{j}_F := -\jump{A\nabla (u - \uiter)\cdot \vn_F}{F}=
\begin{cases}
\jump{A\nabla \uiter \cdot \vn_F}{F}, & \text{ if }F\in \Fz,
\\[3pt]
A\nabla (u - \uiter)\cdot \vn_F , & \text{if }F \subset \p\Omega,
\end{cases}
\end{eqnarray}
respectively. Note that $ {r}_K$ and ${j}_F$ are constants in $K$ and on 
$F$ if $F$ is an interior face, respectively. When $\vz\not \in \partial \Omega$ is 
an 
interior vertex, $\bsig^{\Delta}_{\vz}$ satisfies the following local problem: 
\begin{equation}
\label{eq:loc-f}
\left\{
\begin{aligned}
\divv \bsig^{\Delta}_{\vz} &= \phi_{\vz} {r}_K
- \nabla \phi_{\vz}\cdot \nabla(u- \uiter), & \text{ on } K\subset \omega_{\vz},
\\
\jump{ \bsig^{\Delta}_{\vz}\cdot \vn_F}{F} &= \phi_{\vz} {j}_F, 
& \text{ on } F\in \Fz,
\\
\bsig^{\Delta}_{\vz}\cdot \vn_F &= 0, & \text{ on } F\subset \p \omega_{\vz}.
\end{aligned}
\right.
\end{equation}
If $\vz \in \partial \Omega$, then the first equation in 
\eqref{eq:loc-f} is unchanged, and the flux jump equations change to
\begin{equation}
\label{eq:loc-fbd}
\left\{
\begin{aligned}
\jump{ \bsig^{\Delta}_{\vz}\cdot \vn_F}{F} &= \phi_{\vz} {j}_F, \quad 
& \text{ on } F\in \Fz \text{ and } F\not\subset  \p \omega_{\vz}\cap \p\Omega,
\\
\bsig^{\Delta}_{\vz}\cdot \vn_F &= 0, 
& \text{ on } F\subset \p \omega_{\vz}\backslash \partial \Omega. \qquad \qquad
\end{aligned}
\right.
\end{equation}


To approximate problem 
\eqref{eq:loc-f}, an approximated correction flux $\bsigtz$ is sought in the 
following broken lowest-order Raviart-Thomas space:  
\begin{equation}
\label{eq:sp-rtb}
\RT^0_{-1,\omega_{\vz}}:=
\Bigl\{ \btau\in \vL^2(\omega_{\vz}):\, 
\btau\at{K}\in \RT^0(K),\; \forall K\subset{\omega_{\vz}} \Bigr\},
\end{equation}
where $\RT^0(K)$ denotes the local lowest-order Raviart-Thomas space on $K$ (see 
\cite{Brezzi-Fortin}).

An explicit procedure called the hypercircle method or equilibration (see 
\cite{Braess-07,Braess06eqres}) is used to construct $\bsigtz$. The correction flux $\bsigtz$ 
satisfies the following problem on an interior vertex patch $\omega_{\vz}$ 
($\vz\not\in 
\partial \Omega$):
\begin{equation}
\label{eq:loc-fa}
\left\{
\begin{aligned}
\divv \bsigtz &= \bar{r}_{K,\vz} + c_{\vz}, & \text{ on } K\subset \omega_{\vz},
\\
\jump{ \bsigtz\cdot \vn_F}{F} &= \bar{j}_{F,\vz}, 
& \text{ on } F\in \Fz,
\\
\bsigtz\cdot \vn_F &= 0, & \text{ on } F\subset \p \omega_{\vz},
\end{aligned}
\right.
\end{equation}
where $\bar{r}_{K,\vz}$ and $\bar{j}_{F,\vz}$ are defined as the $L^2$-projection of 
$\phi_{\vz} {r}_{K}$ and $\phi_{\vz} {j}_F$ onto the constant space of $K$ 
and interior $F$, respectively, for $d=2,3$:
\begin{equation}
\begin{aligned}
\label{eq:est-rz}
\bar{r}_{K,\vz} &:= \Pi_K(\phi_{\vz} {r}_{K}) = 
\frac{1}{d+1}f_K = \frac{1}{d+1}r_{K},
\\[3pt]
\bar{j}_{F,\vz} &:= \Pi_F(\phi_{\vz} {j}_F) 
=\frac{1}{d} \jump{(A\nabla \uiter)\cdot \vn_F}{F} = \frac{1}{d}j_{F}.
\end{aligned}
\end{equation}
When $\vz\in \p \Omega$, $c_{\vz}=0$, and the normal fluxes in \eqref{eq:loc-fa} are 
modified accordingly by \eqref{eq:loc-fbd}.
%

Note that, without $c_{\vz}$, the compatibility condition for \eqref{eq:loc-fa}
is not automatically satisfied, that is,
\begin{equation*}
\sum_{K\subset \omega_{\vz}}\binprod{\bar{r}_{K,\vz} }{1}_{K} -\sum_{F\in \Fz} 
\binprod{\bar{j}_{F,\vz}}{1}_{F} \neq 0,
\end{equation*}
which implies that \eqref{eq:loc-fa} does not have a solution.
To guarantee the existence of a solution to \eqref{eq:loc-fa}, an element-wise 
compensation term $c_{\vz}$ is added on the right hand side of the divergence 
equation in \eqref{eq:loc-fa}. Notice that the normal fluxes are kept unchanged so 
that the final recovered flux can still fulfill the $\vhdiv$-continuity condition of 
the space in \eqref{eq:space-duality}. The $c_{\vz}$ is defined as a constant on 
this 
vertex patch $\omega_{\vz}$ enforcing the 
compatibility condition for \eqref{eq:loc-fa}:
\begin{equation}
\label{eq:loc-compatible}
\sum_{K\subset \omega_{\vz}}\binprod{\bar{r}_{K,\vz} +c_{\vz}}{1}_{K} -\sum_{F\in 
\Fz} 
\binprod{\bar{j}_{F,\vz}}{1}_{F} =0,
\end{equation}
which, together with \eqref{eq:est-rz}, yields for an interior vertex $\vz$
\begin{equation}
\label{eq:loc-c1}
\begin{aligned}
c_{\vz} := & \frac{1}{|\omega_{\vz}|}\left(\sum_{F\in 
\Fz}\binprod{\bar{j}_{F,\vz}}{1}_{F} 
- \sum_{K\subset \omega_{\vz}}\binprod{\bar{r}_{K,\vz} }{1}_{K} \right)
\\
= & \frac{1}{|\omega_{\vz}|}\left(\sum_{F\in \Fz}\binprod{{j}_{F}}{\phi_{\vz}}_{F} 
- \sum_{K\subset \omega_{\vz}}\binprod{{r}_{K} }{\phi_{\vz}}_{K} \right)
\\
= &
\frac{1}{|\omega_{\vz}|}\binprod{A \nabla (u - \uiter) }{\nabla 
\phi_{\vz}}_{\omega_{\vz}}.
\end{aligned}
\end{equation}
With $c_{\vz}$, the solution to \eqref{eq:loc-fa} exists since the compatibility 
condition \eqref{eq:loc-compatible} is met (see \cite{Braess06eqres,Cai12eqrobust}). 
We note that if $\uiter$ solves \eqref{eq:pb-pd} exactly, 
i.e., $\uiter = u_{_\cT}$, then $c_{\vz} = 0$ for an interior vertex by 
\eqref{eq:loc-c1}, and this is a consequence of the Galerkin 
orthogonality.

In the case that $\uiter$ is not an exact solution to problem \eqref{eq:pb-pd}, 
we emphasize again that problem \eqref{eq:loc-fa} is not solvable without the 
presence of $c_{\vz}$. The Galerkin orthogonality,  which occurs as the 
compatibility condition for \eqref{eq:loc-fa} if $\vz\not\in \p \Omega$, is violated 
if $\uiter$ is not the exact finite element approximation. 

We also note that if 
$\vz\in \p \Omega$, the Galerkin orthogonality does not hold either, $\binprod{A 
\nabla u_{_\cT} }{\nabla \phi_{\vz}}  \neq \binprod{f}{\phi_{\vz}} =\binprod{A 
\nabla u}{\nabla \phi_{\vz}}$, since the nodal basis 
$\phi_{\vz}$ is not in the test function space for the discretized problem in 
\eqref{eq:pb-pd}. A direct usage of \eqref{eq:loc-c1} implies $c_{\vz}\neq 0$, yet, 
the degrees of freedom for $\bsigtz$ on the faces on $\p\omega_{\vz}\cap\p\Omega$ 
are 
treated as unknowns in \eqref{eq:loc-fm}, and $c_{\vz}$ is not needed in 
\eqref{eq:loc-fa} on a boundary vertex $\vz\in \p\Omega$.


The flux correction is postprocessed by a minimization procedure locally on 
$\omega_{\vz}$:
\begin{equation}
\label{eq:loc-fm}
\norm{A^{-1/2}\bsigtz}_{0,\omega_{\vz}} = 
\min_{\btau \in \SigmaTz} \norm{A^{-1/2}\btau}_{0,\omega_{\vz}},
\end{equation}
where $\SigmaTz:= \bigl\{\btau\in \RT^0_{-1,\omega_{\vz}}: \, \btau \text{ satisfies 
\eqref{eq:loc-fa}}  \bigr\}$. The element-wise and the global flux corrections 
are then:
\begin{equation}
\label{eq:loc-fd}
\bsigtk: =\sum_{\vz\in \cN_K}\bsigtz 
\quad \text{ and }\quad  
\bsigt := \sum_{\vz\in \cN}\bsigtz.
\end{equation}

Lastly, a compensatory flux $\bsigc$, which is in the globally $\vhdiv$-conforming 
$\RT^0$ space, is then sought using $c_{\vz}$ defined in \eqref{eq:loc-c1} as data:
\begin{equation}
\label{eq:loc-fc}
\divv \bsigc = -\sum_{\vz\in \cN_K} c_{\vz},
 \quad\text{ in any }K\in \cT,
\end{equation}
By the surjectivity 
of the divergence operator from $\RT^0$ to $\cS^0$, the above 
problem has a solution (e.g., \cite{Brezzi-Fortin,Chen2009convergence}). If 
$\bsigc$ is sought by minimizing a weighted $L^2$-norm, with \eqref{eq:loc-fc} being 
a constraint, then it is equivalent to seeking the solution to a mixed finite 
element approximation problem in the $\RT^0$--$\cS^0$ pair.  The energy 
estimate in a weighted $L^2$-norm for $\bsigc$, which bridges it with the algebraic 
error, will be shown later in Lemma \ref{lem:loc-f}.

The recovered flux based on the $\uiter$ is defined as:
\begin{equation}
\label{eq:rec-s}
\bsig_{_\cT}:= - A\nabla \uiter+ \bsigt + \bsigc.
\end{equation}
In practice, only $\bsigt$ is explicitly computed. For explicit local constructions 
of $\bsigt$, we refer the readers to \cite{Cai12eqrobust,Braess06eqres}. 
The $\bsigc$ is here to compensate the change in divergence caused by the correction 
term $c_{\vz}$, and is not needed, nor explicitly computed for the estimator defined 
in \eqref{eq:est-eta}. 

\begin{lemma}
\label{lem:rec-s}
The recovered flux $\bsig_{_\cT} $ is in the conforming finite element subspace of 
the duality space: $\bsig_{_\cT} \in\bm{\Sigma}_{\subcT}:= \bm{\Sigma}\cap \RT^0 $.
\end{lemma}

\begin{proof}
Using \eqref{eq:loc-fa} and \eqref{eq:loc-fc}, together with the fact that 
$A \nabla \uiter$ is a constant vector on each element $K$, we have:
\[
\divv \bsig_{_\cT}\at{K}= \divv \bsigt + \divv \bsigc 
= \sum_{\vz\in \cN_K} \bar{r}_{K,\vz} = f_K.
\]
On $F\in \cF$, the continuity of the normal component implies $\bsigc \in 
\vHdiv$ 
\[
\jump{\bsig_{_\cT}\cdot \vn}{F} = \jump{\bsigt\cdot \vn}{F}
-\jump{A \nabla \uiter\cdot \vn}{F} = 
\sum_{\vz\in \cN(F)} \bar{j}_{F,\vz} -  {j}_F = 0.
\]
\end{proof}

\subsection{Discretization error estimator and reliability}

With the recovered flux correction defined in \eqref{eq:loc-fd}, we define 
the discretization error estimator $\eta_d$ as:
\begin{equation}
\label{eq:est-eta}
\eta_{\sub{d,K}} = \big\Vert{A^{-1/2}\bsigtk}\big\Vert_{0,K},
\;\text{ and }\;\eta_d = \big\Vert{A^{-1/2}\bsigt}\big\Vert_{0}.
\end{equation}
The reliability we show in this section is: the total error 
$\norm{u - \uiter}_A$ is bounded by the error estimator $\eta_d$ plus the algebraic 
error. 



In \eqref{eq:loc-c1}, the representation of $c_{\vz}$ uses $u-\uiter$. Nevertheless, 
inserting the Galerkin orthogonality into \eqref{eq:loc-c1}, which reads 
$\binprod{A \nabla (u - u_{_\cT}) }{\nabla \phi_{\vz}}_{\omega_{\vz}}=0$ for any 
interior 
vertex $\vz$, we have
\begin{equation}
\label{eq:est-c2}
c_{\vz}= \frac{1}{|\omega_{\vz}|} 
\binprod{A \nabla (u_{_\cT} - \uiter) }{\nabla \phi_{\vz}}_{\omega_{\vz}}
= \frac{1}{|\omega_{\vz}|} \sum_{K\subset\omega_{\vz}}
\binprod{A \nabla (u_{_\cT} - \uiter) }{\nabla \phi_{\vz}}_{K}.
\end{equation}
Now the compatibility compensation term $c_{\vz}$ 
can be decomposed as follows:
\begin{equation}
\label{eq:loc-c2}
c_{\vz} = \sum_{K\subset \omega_{\vz}}c_{\vz,K}, \;\text{ with }\; c_{\vz,K} := 
\frac{1}{|\omega_{\vz}|}\binprod{A \nabla (u_{_\cT} - \uiter) }{\nabla 
\phi_{\vz}}_{K}.
\end{equation}

\begin{lemma}[Nodal estimate for the compensation term]
\label{lem:loc-r}
For any interior vertex $\vz\in \cN_K$, on $K\subset \omega_{\vz}$, $c_{\vz,K}$ 
satisfies 
the following $L^2$-estimate with $C$ depending on the shape regularity of the patch 
$\omega_{\vz}$:
\begin{equation}
\label{eq:est-c}
h_K  A_{K}^{-1/2}\norm{c_{\vz,K}}_{0,K} \leq C  
 \, \norm{u_{_\cT} - \uiter}_{A,K},
\end{equation}
\end{lemma}

\begin{proof}
By the representation in \eqref{eq:loc-c2}, it follows from the Cauchy-Schwarz 
inequality, the fact that $\norm{\nabla \phi_{\vz}}_{0,K} \leq C\,h_K^{\frac{d}{2} - 
1} $, and the shape regularity of the patch that
\begin{eqnarray*}
\abs{c_{\vz,K}} 
&=& \frac{1}{|\omega_{\vz}|} 
\abs{\binprod{A \nabla (u - \uiter) }{\nabla \phi_{\vz}}_{K}} 
\leq  \frac{1}{|\omega_{\vz}|} \norm{u - \uiter}_{A,K}\norm{\phi_{\vz}}_{A,K}\\[2mm]
&\leq & C \,h_K^{-\frac{d}{2} - 1} 
A_K^{1/2} \,\norm{u - \uiter}_{A,K}.
\end{eqnarray*}
Since $c_{\vz,K}$ is a constant on $K$, $\norm{c_{\vz}}_{0,K} \leq h_K^{\frac{d}{2}} 
\abs{c_{\vz,K}} $, the validity of \eqref{eq:est-c} is then verified.
\end{proof}

To bridge the energy estimate for $\bsigc$ with the algebraic error, the following 
norms (\cite[Chapter 3 $\S 
5.7$]{Braess-07}) are need: let $A_F := \max_{K\subset \omega_F} A_K$, for $p\in 
\cS^0$, and $f\in L^2(\Omega)$
\begin{equation}
\label{eq:norm-h}
\norm{f}_{-1,h} :=\sup_{q\in \cS^0} \frac{(f,q)}{\norm{q}_{1,h}},
\text{ and } \norm{p}_{1,h}:= \displaystyle 
\left(\sum_{F\in \cF} 
h_F^{-1} A_F\norm{\jump{p}{}}_{0,F}^2\right)^{1/2}.
\end{equation}

\begin{lemma}[A discrete energy estimate for $\bsigc$]
\label{lem:loc-f}
If $\bsigc$ is obtained by 
\begin{equation} 
\label{eq:est-fc}
\big\Vert{A^{-1/2}\bsigc}\big\Vert_{0} = \min_{\substack{\btau\in \RT^0,\\
\divv \,\btau = f^c}}
\big\Vert{A^{-1/2}\btau}\big\Vert_{0}, 
\end{equation}
where $f^c$ is defined as follows on an element $K$ using \eqref{eq:loc-fc}, 
\begin{equation}
\label{eq:fcK}
f^c|_K := -\sum_{\vz\in \cN_K} c_{\vz} 
\end{equation}
then the following estimate holds:
\begin{equation}
\label{eq:est-sigmac}
\big\Vert{A^{-1/2}\bsigc}\big\Vert_{0} \leq \, 
C_A\big\Vert{u_{_\cT} - \uiter}\big\Vert_A,
\end{equation}
in which $C$ depends on the shape regularity of the triangulation, the 
maximum number of elements in each $\omega_K$, and the diffusion 
coefficient $A$. 
\end{lemma}
\begin{proof}
The minimizer of problem \eqref{eq:est-fc} satisfies the following global mixed 
problem: find $(\bsigc, p)\in \RT^0\times \cS^0$
\begin{equation}
\label{eq:loc-mixed}
\left\{
\begin{aligned}
&\binprod{A^{-1} \bsigc}{\btau} - \binprod{p}{\divv \btau} = 0, 
\quad \forall\; \btau \in\RT^0,
\\[2pt]
&\binprod{\divv \bsigc}{q} = \binprod{f^c}{q}, \quad \forall \;q\in \cS^0.
\end{aligned}
\right.
\end{equation}
By the inf-sup stability of discrete 
$H^1$-$\vL^2$ analysis of the mixed problem when the shape regularity of the mesh is 
assumed ($ h_F \eqsim h_K$ for $F$'s neighboring elements) (\cite[Chapter 3 $\S 
5.7$]{Braess-07}), 
problem \eqref{eq:loc-mixed} has a unique solution satisfying the 
following energy estimate: letting $\btau = \bsigc$, $q = p$, we have
\begin{equation}
\begin{aligned}
&\big\Vert{A^{-1/2}\bsigc}\big\Vert_{0}^2\leq \norm{f^c}_{-1,h} \norm{p}_{1,h}
\leq \norm{f^c}_{-1,h} \sup_{\btau\in \RT^0 } 
\frac{(p,\nabla\cdot\btau)}{\norm{A^{-1/2}\btau}_{0}}
\\
= &  \norm{f^c}_{-1,h} \sup_{\btau\in \RT^0 } 
\frac{\binprod{A^{-1} \bsigc}{\btau}}{\norm{A^{-1/2}\btau}_{0}} 
\leq 
\norm{f^c}_{-1,h} \big\Vert{A^{-1/2}\bsigc}\big\Vert_{0}.
\end{aligned}
\end{equation}
Now, to prove the validity of the lemma, by \eqref{eq:norm-h}, it suffices to show 
that for $q\in \cS^0$
\begin{equation}
(f^c,q) \leq C\norm{u_{_\cT} - \uiter}_A \norm{q}_{1,h}.
\end{equation}
To this end, first denote $q_K:=q|_K$, and $f_c$ is written out explicitly using 
\eqref{eq:fcK}, 
\begin{equation}
(f^c,q)  = -\sum_{K\in \cT}\left( \sum_{\vz\in \cN_K} 
c_{\vz} ,q\right)_{K}
= -\sum_{K\in \cT}\sum_{\vz\in \cN_K} c_{\vz} q_K|K|.
\end{equation}
Using $c_{\vz} = \sum_{K\subset \omega_{\vz}}c_{\vz,K}$ in \eqref{eq:loc-c2} for 
interior 
vertices and $c_{\vz} =0 $ for $\vz\in \p\Omega$ yields,
\begin{equation}
\label{eq:sum-1}
(f^c,q) = -\sum_{K\in \cT}\sum_{\vz\in \cN_K,\vz\not\in \partial\Omega} 
\left( \sum_{T\subset\omega_{\vz}} c_{\vz,T}\right)q_K|K|.
\end{equation}
We switch the order of the summation, by summing up the 
inner terms $c_{\vz,T}$ last, then the above equation becomes
\begin{equation}
\label{eq:sum-2}
\begin{aligned}
& -\sum_{K\in \cT}\sum_{\vz\in \cN_K,\vz\not\in \partial\Omega} 
\left( \sum_{T\subset\omega_{\vz}} c_{\vz,T}\right)q_K|K|
\\
=& \, -\sum_{K\in \cT} \sum_{\vz\in \cN_{K},\vz\not\in \partial\Omega}
\left\{   c_{\vz,K} \Big(\sum_{T\subset \omega_{\vz}} q_T|T| \Big)
\right\} =:-{(*)} ,
\end{aligned}
\end{equation}
in which for each vertex $\vz\in \cN_{K}$, the term $c_{\vz,K}$ is only summed 
against $q_T|T|$ for $T\subset \omega_{\vz}$. The reason is that among the terms in 
the 
original summation in \eqref{eq:sum-1}, a term involving $c_{\vz,T}$ is summed up 
multiplying $q_K|K|$ only when $\omega_{\vz}\subset \omega_K$. 

Now on each $K$ not touching $\partial \Omega$, we have the following weighted 
average 
of $c_{\vz,K}$, using $|\omega_{\vz}| m_K$ as weights, being zero for any $m_{K}$ 
that is a constant on the patch $\omega_K$:
\begin{equation}
\sum_{\vz\in \cN_{K},\vz\not\in \partial\Omega}  c_{\vz,K}  (|\omega_{\vz}| m_K )= 
m_K\sum_{\vz\in \cN_{K},\vz\not\in \partial\Omega}
\binprod{A \nabla (u_{_\cT} - \uiter) }{\nabla \phi_{\vz}}_{K} = 0.
\end{equation}
As a result, $|\omega_{\vz}| m_K$ can be inserted into \eqref{eq:sum-2}, and $m_K$ 
is chosen 
as the average of $q$ on $\omega_K$, i.e., $m_K:= (\sum_{P\subset \omega_K} q_{P}|P| 
)/|\omega_K|$, thus $(*)$ in \eqref{eq:sum-2} becomes
\begin{equation}
\label{eq:sum-3}
\begin{aligned}
&\sum_{K\in \cT} \sum_{\vz\in \cN_{K},\vz\not\in \partial\Omega}
\left\{  c_{\vz,K} \Big(\sum_{T\subset \omega_{\vz}} q_T|T|
-\frac{|\omega_{\vz}|}{|\omega_K|}\sum_{P\subset \omega_K} q_P|P| \Big)
\right\} 
\\
=& \sum_{K\in \cT} \sum_{\vz\in \cN_{K},\vz\not\in \partial\Omega}
\left\{ c_{\vz,K} \sum_{T\subset \omega_{\vz}}\Big( \frac{|T|}{|\omega_K|} 
\sum_{P\subset \omega_K}  (q_T -q_P)|P|  \Big)
\right\}  =: \sum_{K\in \cT} \sum_{\vz\in \cN_{K},\vz\not\in \partial\Omega} 
\beta_{\vz,K}.
\end{aligned}
\end{equation}
For any $T\subset \omega_{\vz}$, if $T$ and $P\subset \omega_K$ 
have a common face $F=\partial T\cap \partial P$,
$|q_T -q_P| = |\jump{q}{F}|$ on $F$; 
otherwise, there always exists a path consisting of 
finite many elements 
$K_i\subset \omega_K$ ($i=1,\dots, n_{TP}$) starting from $K_1:= T$ to 
$K_{n_{TP}}:=P$, 
such that $K_i$ and $K_{i-1}$ share a face $F_i$, then 
\begin{equation}
|q_T -q_P| = \left|q_{K_1} -q_{K_2}  + q_{K_2} -q_{K_3} \dots \right|
\leq \sum_{i=1}^{n_{TP}}\abs{\jump{q}{F_i}} \leq \sum_{F\in \cF_{\omega_K}} 
\abs{\jump{q}{F}}.
\end{equation} 
Applying above on the innermost summation for $P$ of 
\eqref{eq:sum-3}, exploiting the local shape regularity on every element in 
$\omega_K$, and using the fact that $c_{\vz,K}$ and $\jump{q}{F}$ are constants on 
$K$ 
and $F$, respectively, yields:
\begin{equation}
\begin{aligned}
\beta_{\vz,K} 
& \leq |c_{\vz,K}|\abs{\sum_{T\subset \omega_{\vz}}\Big( \frac{|T|}{|\omega_K|} 
\sum_{P\subset \omega_K}  (q_T -q_P)|P| \Big)}
\leq |c_{\vz,K}| \left(\sum_{T\subset \omega_{\vz}} |T| \sum_{F\in \cF_{\omega_K}} 
\abs{\jump{q}{F}}\right)
\\
&\leq C   A_K^{-1/2}h_K\norm{c_{\vz,K}}_{0,K} \cdot 
A_K^{1/2}h_K^{-1}|K|^{1/2}\left(\sum_{F\in \cF_{\omega_K}} 
\norm{\jump{q}{F}}_{0,F} |F|^{-1/2}\right).
\end{aligned}
\end{equation}

Using the Cauchy-Schwarz inequality and the shape regularity of the 
triangulation, $(*)$ can be estimated as follows:
\begin{equation}
\begin{aligned}
(*) \leq C & \left(\sum_{K\in \cT} \sum_{\vz\in \cN_{K},\vz\not\in \partial\Omega}
A_K^{-1}h_K^2\norm{c_{\vz,K}}_{0,K}^2\right)^{1/2} 
\\
&\left(\sum_{K\in \cT} \sum_{\vz\in \cN_{K},\vz\not\in \partial\Omega}
A_K \sum_{F\in \cF_{\omega_K}} h_F^{-1}\norm{\jump{q}{F}}_{0,F}^2\right)^{1/2}.
\end{aligned}
\end{equation}
Finally, the lemma follows from Lemma \ref{lem:loc-r} and definition 
\eqref{eq:norm-h}.
\end{proof}

\begin{theorem} 
\label{thm:r}
There exists a positive constant $C_A$, depending on the shape regularity of the mesh and the coefficient $A$,
such that
\begin{equation}
\label{eq:est-r}
\norm{u-\uiter}_A \leq \eta_d + C_A \norm{u_{_\cT} - \uiter}_{A}.
\end{equation}
\end{theorem}

\begin{proof}
The proof of \eqref{eq:est-r} starts from \eqref{eq:est-u}
\[
\norm{u-\uiter}_A ^2 
\leq 2\Big( \cJ(\uiter) - \cJ^*(\bsig_{_\cT}) \Big) =   \norm{A^{1/2}\nabla \uiter }_0^2 
- 2\binprod{f}{\uiter}
+ \binprod{A^{-1} \bsig_{_\cT}}{\bsig_{_\cT}}.
\]
With $\bsig_{_\cT}=-A \nabla \uiter + \bsigt + \bsigc$ defined in \eqref{eq:rec-s}, we have
\[
 \binprod{A^{-1} \bsig_{_\cT}}{\bsig_{_\cT}} 
= \norm{A^{-1/2}(\bsigt + \bsigc)}_0^2 
- 2 \binprod{\bsig_{_\cT}}{\nabla \uiter} - \norm{A^{1/2}\nabla \uiter }_0^2,
\]
which, together with the above inequality, implies
 \begin{eqnarray*}
 \norm{u-\uiter}_A ^2
& \leq & \norm{A^{-1/2}(\bsigt + \bsigc)}_0^2  -2\binprod{\bsig_{_\cT}}{\nabla \uiter} - 2\binprod{f}{\uiter} \\[2mm]
 &=&  \norm{A^{-1/2}(\bsigt + \bsigc)}_0^2.
 \end{eqnarray*}
The last equality uses the fact that $\binprod{\bsig_{_\cT}}{\nabla \uiter} + 
\binprod{f}{\uiter}=0$, which follows
from integration by parts element-wise and Lemma \ref{lem:rec-s}. 
By the triangle inequality, we have
\[
\norm{u-\uiter}_A 
 \leq \norm{A^{-1/2}\bsigt }_0 + \norm{A^{-1/2}\bsigc}_0
=\eta_d + \norm{A^{-1/2}\bsigc}_0.
\]
Now, the theorem simply follows from estimate \eqref{eq:est-sigmac} in 
Lemma \ref{lem:loc-f}.
\end{proof}

\section{Algebraic error estimator}

The upper bound in \eqref{eq:est-r} contains the algebraic error 
$\norm{u_{_\cT} - \uiter}_{A}$. This section introduces an algebraic error estimator in terms of
the energy norm of two consecutive iterates with a constant depending on an 
approximation of the spectral radius of the error propagation matrix.

Recall the stiffness matrix $\bA$ introduced in Section 2 and the iteration 
in \eqref{eq:pb-it}.
Denote the algebraic iteration error at the $k$-th iteration by
\begin{equation}
\label{eq:err-a}
\mathbf{e}^{(k)} : = \bu_{_\cT} - \bu_{_\cT}^{(k)},
\end{equation}
then the error propagation can be verified to be:
\begin{equation}
\label{eq:err-a1}
\mathbf{e}^{(k+1)} = (\bI - \bB \bA)\mathbf{e}^{(k)}.
\end{equation}
Let $e^{(k)}$ be the function in the finite element space having $\mathbf{e}^{(k)}$ 
as its 
vector representation in the nodal basis. Define the spectral radius of the error 
propagation matrix $\bI - \bB \bA$ as 
$\rho_{\mathrm{err}}$:
\begin{equation}
\rho_{\mathrm{err}} : = \rho(\bI - \bB \bA) = \norm{\bI - \bB \bA}_2.
\end{equation}

\begin{theorem}[Upper bound of the algebraic error]
\label{thm:ra}
Let $\{\bu^{(k)}\}$ be the sequence generated by \eqref{eq:pb-it}, 
then the algebraic error $\mathbf{e}^{(k)}$ defined in \eqref{eq:err-a} satisfies
the following estimate:
\begin{equation}
\label{eq:est-au}
\big\Vert{\mathbf{e}^{(k+1)}}\big\Vert_{\bA}
\leq \frac{\rho_{\mathrm{err}}}{1-\rho_{\mathrm{err}}} 
\big\Vert{\bu^{(k+1)} - \bu^{(k)}}\big\Vert_{\bA}, 
\end{equation}
or in the finite element function form:
\begin{equation}
\big\Vert{u_{_\cT} -  u_{_\cT}^{(k+1)}}\big\Vert_A \leq 
\frac{\rho_{\mathrm{err}}}{1-\rho_{\mathrm{err}}} 
\big\Vert{u_{_\cT}^{(k+1)}-  u_{_\cT}^{(k)}}\big\Vert_{A}.
\end{equation}
\end{theorem}

\begin{proof}
By the norm equivalence in \eqref{eq:nm-ma} and the fact that $(\bI -\bA^{1/2}
\bB \bA^{-1/2})$ is similar to $(\bI - \bB \bA)$ (they have the same
eigenvalues), we have
\begin{equation}
\norm{\bI - \bB \bA}_{\bA} = 
\big\Vert{\bA^{1/2}(\bI - \bB \bA)\bA^{-1/2}}\big\Vert_{2}
= \rho\big(\bI -\bA^{1/2} \bB \bA^{-1/2}\big) = \rho_{\mathrm{err}} .
\end{equation}
Hence, $\norm{\mathbf{e}^{(k+1)}}_{\bA} \leq \rho_{\mathrm{err}}
\norm{\mathbf{e}^{(k)}}_{\bA} $, and the result follows from a standard
contraction mapping convergence theorem (see, e.g., \cite[Theorem 
12.1.2]{Ortega-Rheinboldt}).
\end{proof}

In Theorem \ref{thm:ra}, $\rho_{\mathrm{err}}$ is the true rate of convergence of 
the solver. However, in practice, $\rho_{\mathrm{err}}$ is not available during 
any iteration of the solver, unless an eigenvalue problem is solved for the 
error propagation matrix $\bI - \bB \bA$. What we have access to is the following 
quantity:
\begin{equation}
\rho_{\mathrm{err}}^{(k)} := 
\frac{\norm{\mathbf{r}_{k}}_2}{\norm{\mathbf{r}_{k-1}}_2},
\end{equation} 
where $\mathbf{r}_k:= \bA \mathbf{e}^{(k)} $ with $j$-th entry given by
$(f,\phi_{\vz_j}) - \binprod{A \nabla u_{_\cT}^{(k)}}{\nabla \phi_{\vz_j}}$. The 
following lemma describes the convergence of $\rho_{\mathrm{err}}^{(k)}$ provided that 
the iterative 
solver is convergent.

\begin{lemma}[Convergence of $\rho_{\mathrm{err}}^{(k)}$]
\label{lem:rho-c}
Assuming the error propagation matrix $\bI - \bB \bA$ has eigenvalues $1> 
\rho_{\mathrm{err}} = 
\lambda_1 \geq \lambda_2 \geq \cdots \geq \lambda_N > 0$, then 
$\rho_{\mathrm{err}}^{(k)} \to 
\rho_{\mathrm{err}}$ as $k\to \infty$.
\end{lemma}

\begin{proof}
First notice that, by applying \eqref{eq:err-a1} from $0$ to $k$ in a cascading 
fashion, 
\[
\mathbf{r}_k= \bA \mathbf{e}^{(k)}  = \bA (\bI - \bB\bA)^k \mathbf{e}^{(0)}  
= (\bI - \bA\bB)^k \bA \mathbf{e}^{(0)}  =
(\bI - \bA\bB)^k \mathbf{r}_0.
\]
Since $\bA^{-1}(\bI - \bA\bB)\bA = \bI - \bB\bA$, $\bI - \bA\bB$ and $\bI - \bB\bA$
share the same eigenvalues and eigenvectors.
Suppose that $\{ \bv_i\}_{i=1}^N$ are the set of orthonormal eigenvectors in the
$\ell^2$-sense corresponding to the eigenvalue set $\{\lambda_i \}_{i=1}^N$. Let 
$c_i = \mathbf{r}_0 \cdot \bv_i$ be the coefficient of the eigen-expansion of $\mathbf{r}_0$. 
Without 
loss of generality, assume the multiplicity of the largest eigenvalue $\lambda_1$ is 
1. Then we have:
\begin{equation}
\label{eq:rho-c1}
\begin{aligned}
\rho_{\mathrm{err}}^{(k)} & = 
\frac{\norm{(\bI-\bA\bB)^{k}\mathbf{r}_0}_2}{\norm{(\bI-\bA\bB)^{k-1}\mathbf{r}_0 
}_2}
= \frac{\norm{\sum_{i=1}^{N}\lambda_i^{k} c_i \bv_i}_2}
{\norm{\sum_{i=1}^{N}\lambda_i^{k-1} c_i \bv_i}_2}
\\
& = \lambda_1
\frac{\norm{\sum_{i=1}^{N}
\left(\frac{\lambda_i}{\lambda_1}\right)^{k} c_i \bv_i}_2}
{\norm{\sum_{i=1}^{N}
\left(\frac{\lambda_i}{\lambda_1}\right)^{k-1} c_i \bv_i}_2}
= \lambda_1
\frac{1+ \sum_{i=2}^{N} b_i \gamma_i^{k} } {1+ \sum_{i=2}^{N} b_i \gamma_i^{k-1} },
\end{aligned}
\end{equation}
where $b_i := (c_i/c_1)^2$, and $\gamma_i:= (\lambda_i/\lambda_1)^2$. 
The lemma follows from letting $k\to \infty$. When the multiplicity of $\lambda_1$ 
is $m\geq 2$, factoring out the first $m$ terms and $i$ starts from $(m+1)$ in the 
eigen-expansion in \eqref{eq:rho-c1} yields the same result.
\end{proof}

\begin{lemma}[Monotonicity of $\rho_{\mathrm{err}}^{(k)}$]
\label{lem:rho-m}
Under the same assumption as in Lemma \ref{lem:rho-c},  
$\rho_{\mathrm{err}}^{(k)} \leq \rho_{\mathrm{err}}^{(k+1)}$, for any fixed $k \in 
\ZZ^+$.
\end{lemma}

\begin{proof}
By \eqref{eq:rho-c1}, to prove the validity of the lemma, it suffices to show that:
\begin{equation}
\label{eq:rho-m1}
\left(1+\sum_{i=2}^N b_i\gamma_i^k \right)^2 \leq 
\left(1+\sum_{i=2}^N b_i\gamma_i^{k-1} \right) 
\left(1+\sum_{i=2}^N b_i\gamma_i^{k+1} \right),
\end{equation}
which is equivalent to 
\begin{equation}
2 \sum_{i=2}^N b_i\gamma_i^k + \left(\sum_{i=2}^N b_i\gamma_i^k \right)^2
\leq \sum_{i=2}^N b_i\Bigl(\gamma_i^{k-1} + \gamma_i^{k-1}\Bigr) 
+ \left(\sum_{i=2}^N b_i\gamma_i^{k-1} \right) 
\left(\sum_{i=2}^N b_i\gamma_i^{k+1} \right).
\end{equation}
Since $b_i\geq 0$, $\lambda_i\geq 0$, and $2\gamma_i\leq 1+\gamma_i^2$, we have
\[
2 \sum_{i=2}^N b_i\gamma_i^k \leq \sum_{i=2}^N b_i\Bigl(\gamma_i^{k-1} + 
\gamma_i^{k-1}\Bigr) ,
\]
Then it suffices to show the following inequality:
\begin{equation}
\label{eq:rho-m2}
a:=\left(\sum_{i=2}^N b_i\gamma_i^k \right)^2 - \left(\sum_{i=2}^N b_i\gamma_i^{k-1} 
\right) 
\left(\sum_{i=2}^N b_i\gamma_i^{k+1} \right) \leq 0,
\end{equation}
which will be proved by a standard inductive argument. To this end, let $N=2$, it is 
easy to see that \eqref{eq:rho-m2} holds with equality. Next, assume that 
\eqref{eq:rho-m2} holds for $N=n$. For 
$N = n+1$: we have
\begin{equation}
\label{eq:rho-m3}
\begin{aligned}
a &=  \left(\sum_{i=2}^n b_i\gamma_i^k + b_{n+1}\gamma_{n+1}^k\right)^2\!
- \left(\sum_{i=2}^n b_i\gamma_i^{k-1}+ b_{n+1}\gamma_{n+1}^{k-1}\right) \!
\left(\sum_{i=2}^n b_i\gamma_i^{k+1} + b_{n+1}\gamma_{n+1}^{k+1}\right) 
\\
\quad & \leq \left(\sum_{i=2}^n b_i\gamma_i^{k} \right) ^2
-\left(\sum_{i=2}^n b_i\gamma_i^{k-1} \right) \left(\sum_{i=2}^n b_i\gamma_i^{k+1} 
\right) 
- b_{n+1}\gamma_{n+1}^{k-1} \sum_{i=2}^n b_i(\gamma_{n+1} - 
\gamma_i)^2\gamma_i^{k-1}.
\end{aligned}
\end{equation}
Now \eqref{eq:rho-m2} is a direct consequence of the induction hypothesis. This 
completes the proof of the lemma.
\end{proof}


After the preparation, now we define the algebraic error estimator as follows at the 
$(k+1)$-th iteration of the solver: for $k\geq 1$

\begin{equation}
\label{eq:est-alg}
\eta_{a}^{(k+1)} 
:=e^{1/k}\frac{\rho_{\mathrm{err}}^{(k)}}{1-\rho_{\mathrm{err}}^{(k)}} 
\big\Vert{\bu^{(k+1)} - \bu^{(k)}}\big\Vert_{\bA} = 
e^{1/k}\frac{\rho_{\mathrm{err}}^{(k)}}{1-\rho_{\mathrm{err}}^{(k)}} 
\big\Vert{u_{_\cT}^{(k+1)} - u_{_\cT}^{(k)}}\big\Vert_A.
\end{equation}
The $e^{1/k}$ factor is added to remedy the fact that 
$\rho_{\mathrm{err}}^{(k)}$ converges to $\rho_{\mathrm{err}}$ from 
below.
Without it, the solver might stop too early, before a good estimate of
$\rho_{\mathrm{err}}$ is obtained.

\begin{theorem}[Reliability of the algebraic error estimator]
\label{thm:rak}
Under the same setting with Theorem \ref{thm:ra} and Lemma \ref{lem:rho-c}, there 
exists an $N\in \ZZ^+$ such that for all $k\geq N$, 
\begin{equation}
\big\Vert{\mathbf{e}^{(k+1)}}\big\Vert_{\bA} = 
\big\Vert{u_{_\cT} - u_{_\cT}^{(k+1)}}\big\Vert_A \leq 
\eta_{a}^{(k+1)} .
\end{equation}
\end{theorem}

\begin{proof}
Denote $p(k): = \rho_{\mathrm{err}}^{(k)}$, $\xi(k) := 
p(k)/\Bigl(1-p(k)\Bigr)$, 
and $\xi:=\rho_{\mathrm{err}}/(1-\rho_{\mathrm{err}}) $. By Theorem 
\ref{thm:ra}, it suffices to show that: there 
exists an $N$ such that for $k\geq N$
\begin{equation}
\label{eq:rho-k1}
\xi \leq e^{1/k} \xi(k) .
\end{equation}
It is straightforward to verify that $\xi(k)\to \xi$ from below as $p(k) \to 
\rho_{\mathrm{err}}$. Moreover, $e^{1/k} \xi(k) \to \xi$ as $k\to \infty$. 
Now it suffices to show that when $k$ is sufficiently large, $e^{1/k} \xi(k) $ is a 
decreasing function of $k$. Recalling from \eqref{eq:rho-c1} in Lemma 
\ref{lem:rho-c} that if $k$ is sufficiently large, 
\begin{equation}
\label{eq:rho-k2}
p(k) =\rho_{\mathrm{err}}^{(k)} \simeq \rho_{\mathrm{err}} \frac{1+ b 
\gamma^{k}}{1 + b \gamma^{k-1}},
\end{equation}
where $\gamma := (\lambda_{m+1}/\lambda_1)^2 < 1$, and $b := (c_{m+1}/c_1)^2 \geq 0$ 
where $m$ is the multiplicity of the largest eigenvalue. Taking the derivative of 
$e^{1/k} \xi(k)$ with respect to $k$ leads to:
\begin{equation}
\label{eq:rho-k3}
\frac{d}{dk}\big(e^{1/k} \xi(k)\big) 
= e^{1/k} 
\frac{-k^{-2}p(k)\big(1 - p(k)\big) + p'(k)}{\big(1 - p(k)\big)^2}.
\end{equation}
By \eqref{eq:rho-k2}, we have
\begin{equation}
\label{eq:rho-k4}
p'(k) \simeq \rho_{\mathrm{err}} \frac{b(\gamma-1)\gamma^{k-1}\ln \gamma }
{\big(1 + b \gamma^{k-1}\big)^2}
= O(\gamma^{k-1}).
\end{equation}
Using \eqref{eq:rho-k4} in \eqref{eq:rho-k3} and noting that $k^{-2}$ decreases
at a slower rate than $\gamma^{k-1}$, then for sufficiently large $k$,
$\frac{d}{dk}(e^{1/k} \xi(k))< 0$, and the theorem follows.
\end{proof}

\begin{remark}[Speed up of the rate of convergence estimate]
We notice that without the correction factor in \eqref{eq:est-alg}, the closer 
$\rho_{\mathrm{err}}^{(k)}$ is to $\rho_{\mathrm{err}}$, the more accurate the 
algebraic estimator is. The convergence of $\rho_{\mathrm{err}}^{(k)}$ can be 
accelerated in the following way: 
\[
\rho_{\mathrm{err}}^{(k)}  \approx
\rho_{\mathrm{err}} \frac{1+ b \gamma^{k}}{1 + b \gamma^{k-1}}\quad,\text{ and }\quad
\rho_{\mathrm{err}}^{(k-1)}  \approx
\rho_{\mathrm{err}} \frac{1+ b \gamma^{k-1}}{1 + b \gamma^{k-2}}.
\]
Now we define for $k\geq 2$:
\begin{equation}
\widehat{\rho}_{\mathrm{err}}^{(k)} : = \rho_{\mathrm{err}}^{(k)} 
\frac{\rho_{\mathrm{err}}^{(k)}}{\rho_{\mathrm{err}}^{(k-1)}}.
\end{equation}
And $\widehat{\rho}_{\mathrm{err}}^{(k)}$ converges to $\rho_{\mathrm{err}}$ faster 
than 
the original $\rho_{\mathrm{err}}^{(k)}$. To see this, taking derivative of 
$\widehat{\rho}_{\mathrm{err}}^{(k)}$ with respect to $k$ gives,
\[
\frac{d}{dk}\big(\widehat{\rho}_{\mathrm{err}}^{(k)}\big)= O(\gamma^{k-2}),
\]
which is an order faster than the convergence of $\rho_{\mathrm{err}}^{(k)}$ in 
\eqref{eq:rho-k4}.
\end{remark}

\section{Discretization-accurate stopping criterion}
\label{sec:stopping}

Identity (\ref{eq:est-id}) clearly indicates that the iterative solver should be
stopped when the algebraic error is of the same magnitude as the discretization
error. This observation suggests the following stopping criterion: let 
$\eta_{d}^{(k)}$ be $\eta_{d}$ from \eqref{eq:est-eta} computed using the iterate 
$u_{_\cT}^{(k)}$, the iterative solver shall stop when
\begin{equation}
\label{alg:stop}
\eta_{a}^{(k)} < \varepsilon^{-1} \cdot\eta_{d}^{(k)} \; \text{ and }\;   
\abs{\rho_{\mathrm{err}}^{(k)}/\rho_{\mathrm{err}}^{(k-1)}-1 }< 
\varepsilon_{\rho},  
\end{equation}
where $\varepsilon=\eta_d/\norm{u - u_{\sub{\cT}}}_A$ is the effectivity index.
In light of the proof of Theorem~\ref{thm:rak}, the second condition implies that $\rho_{\mathrm{err}}^{(k)}$ 
is a good approximation to $\rho_{\mathrm{err}}$ and, hence, $\eta^{(k)}_{a}$ is an accurate representation of 
the algebraic error $\norm{u_{_\cT} - u_{_\cT}^{(k)}}_A$ at the $k$-th iteration. Together with the 
first condition, the estimated algebraic error is of the same magnitude in the 
discretization error. 

\section{Numerical examples}
\label{sec:numex}

In this section, several examples are presented to verify the reliability of 
the estimators proposed, as well as the stopping criterion.
The error estimator $\eta_{d}$, using a localized equilibrated flux to solve
\eqref{eq:loc-fa}, is implemented and publicly available in $i$FEM 
\cite{Chen.L2008c}. The initial guess 
for all examples presented in this section is a random guess with each entry of 
$\bu^{(0)}$ satisfying a uniform distribution in $[-1,1]$ using a fixed seed.
An effectivity index of $\varepsilon=1.5$ or $\varepsilon^{-1}=2/3 \approx 0.67$
is used in \eqref{alg:stop}.  This is similar to typical values used in practice
when $u_{\sub{\cT}}$ is computed with a direct solver.

The first test problem is the Poisson equation 
\[
-\Delta u = f, \mbox{ in } \Omega =(-1,\,1)^2
\] 
with Dirichlet boundary conditions 
and the exact solution is given by
\[
u = \alpha\Big( \sin(\pi x)\sin(\pi y) + 0.5\sin(4\pi x)\sin(4\pi y) \Big),
\]
where the constant $\alpha$ is chosen such that $\norm{u}_A = 1$.
This problem is discretized by the continuous piecewise linear finite element method 
on a uniform triangular mesh with mesh size $h=1/32$. 

The resulting system of algebraic equations is first solved by a multigrid method with
$V(1,1)$-cycle. 
Convergence of the multigrid solver in the energy norm along with the algebraic 
estimator are depicted in Figure \ref{fig:ex1-conv} (see the red and blue dot-circle 
lines), 
which numerically verify Theorem \ref{thm:rak} for the algebraic estimator $\eta_a$ being an upper bound of the 
algebraic error.
The total and the discretization errors along with the discretization estimator are also depicted in Figure \ref{fig:ex1-conv} 
(see the red solid-diamond, the red dot, and the blue solid-diamond lines, respectivley).
Estimated convergence rates based on both $\rho_{\mathrm{err}}^{(k)}$ and $\hat{\rho}_{\mathrm{err}}^{(k)}$
are presented to numerically verify Remark~9.

Using the first stopping criterion in \eqref{alg:stop} with $\varepsilon^{-1} = 
0.67$, the multigrid iteration 
stops after merely two iterations, and Figure \ref{fig:ex1-conv} shows that 
the algebraic error already drops below the discretization error. One conventional 
stopping criterion, which is usually integrated in high performance computing 
packages such as \emph{hypre} \cite{Falgout2002hypre}, uses the relative residual 
measured in the $\ell^2$-norm: when setting $\norm{\bA \mathbf{e}^{(k)}}_0/\norm{\bA 
\mathbf{e}^{(0)}}_0\leq 10^{-7}$, the multigrid iteration stops after fifteen 
iterations.
For a slower iterative solver, we also implement symmetric 
Gauss-Seidel iterative method. The first stopping criterion in \eqref{alg:stop} with 
$\varepsilon^{-1} = 0.67$
requires only thirty-one iterations, while the conventional stopping criterion with the tolerance $10^{-5}$
needs more than two hundred eighty iterations. 
These results show a dramatic reduction in computational cost
when using the discretization-accurate stopping criterion introduced in this paper.
The numbers of iterations for the multigrid and the symmetric Gauss-Seidel iterative methods with
both the stopping criterions as well as the total and the algebraic errors 
are summarized in Table~\ref{table:ex-comp}. As observed from Table~\ref{table:ex-comp}, 
additional iterations needed by the conventional stopping criterion significantly 
decrease the algebraic errors but not the total errors.
Figure \ref{fig:ex1-soln} compares the solution $u_{_\cT}$ obtained by a direct 
solver
with that of a multigrid solver after 2 iterations.

\begin{table}[h]
\caption{The number of iterations and the total and algebraic errors for the Poisson problem.} 
\centering  
\begin{tabular}{|c|c|c|c|c|} 
\hline & Stopping  &\# Iter 
& $\norm{u - {u}_{_\cT}^{(k)}}_A$ & $\norm{u_{_\cT} - u_{_\cT}^{(k)}}_A$ 
\\[1mm]
\hline
MG V(1,1)  &  $\eta_a\leq 0.67 \eta_d$ & 2 & $0.0821$ & $3.5\times 10^{-1}$  
\\[1mm]
\hline
MG V(1,1)  &  ${\norm{\mathbf{r}_{k}}_2}/{\norm{\mathbf{r}_0}_2}\leq 10^{-7}$ & 
15 & $0.0741$  &   $3.4\times 10^{-8}$
\\[1mm]
\hline
Sym GS  & $\eta_a\leq 0.67 \eta_d$ & 31  &  $0.1051$  & 
$7.5\times 10^{-1}$   
\\[1mm]
\hline
Sym GS  &  ${\norm{\mathbf{r}_{k}}_2}/{\norm{\mathbf{r}_0}_2}\leq 10^{-5}$ & 
 289 &  $0.0741$   &  $2.7\times 10^{-4}$
\\[1mm]
\hline
\end{tabular}
\label{table:ex-comp}
\end{table}

\begin{figure}[h]
\centering
\begingroup
\captionsetup[subfigure]{width=0.45\textwidth}
\subfloat[The convergence of the $V(1,1)$-cycles.]
{
\includegraphics[height=0.25\textheight]{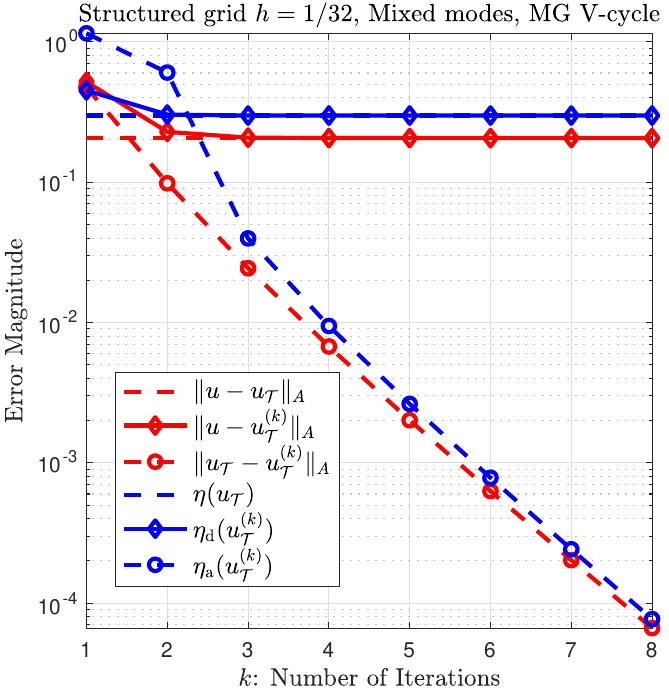}
\label{fig:ex1-conv}
}
\qquad\qquad
\vspace{0.1in}
\subfloat[The convergence of the estimated rate of convergence]
{
\includegraphics[height=0.25\textheight]{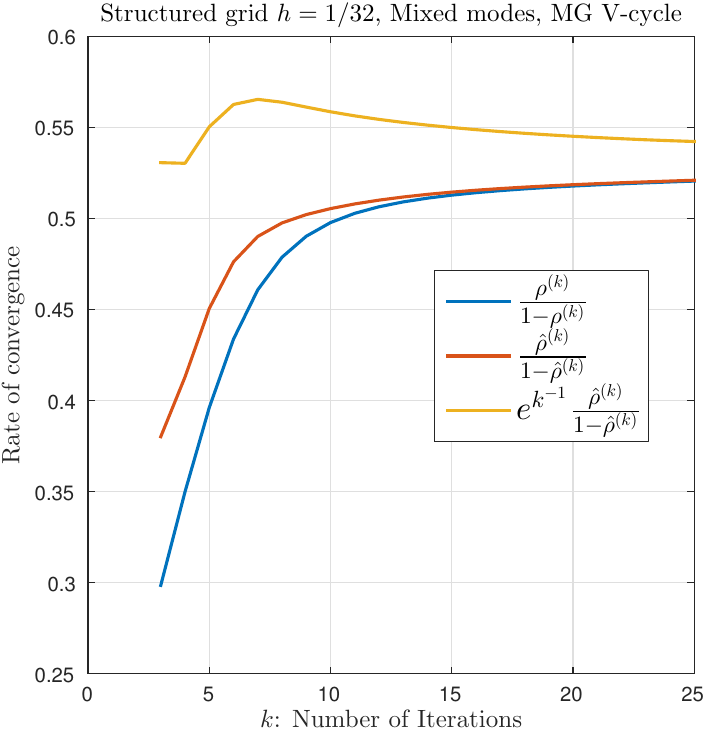}
\label{fig:ex1-rho}
}
\endgroup
\caption{The convergence results for the Poisson problem: the solution has mixed modes, the 
problem is discretized on a uniform triangular mesh, and the linear system is 
approximated using $V(1,1)$-cycle iterations.}
\end{figure}

\begin{figure}[h]
\centering
\begingroup
\subfloat[$u_h$ obtained by a direct solver.]
{
\includegraphics[width=0.4\textwidth]{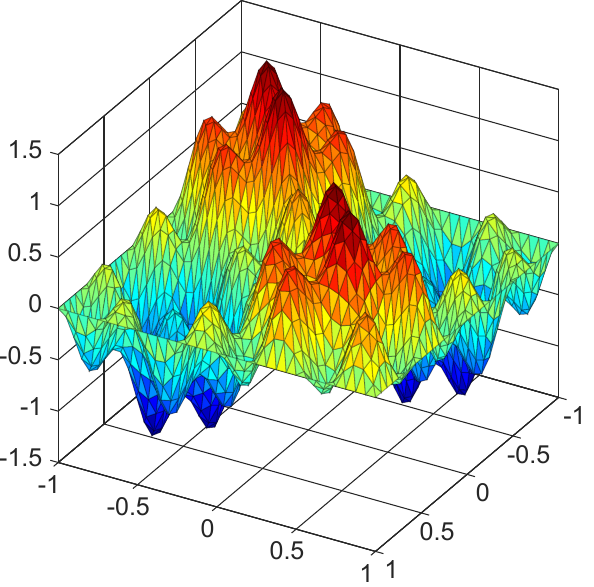}
}
\qquad\qquad
\vspace{0.1in}
\subfloat[$u_h^{(2)}$ obtained by two $V(1,1)$-cycles.]
{
\includegraphics[width=0.4\textwidth]{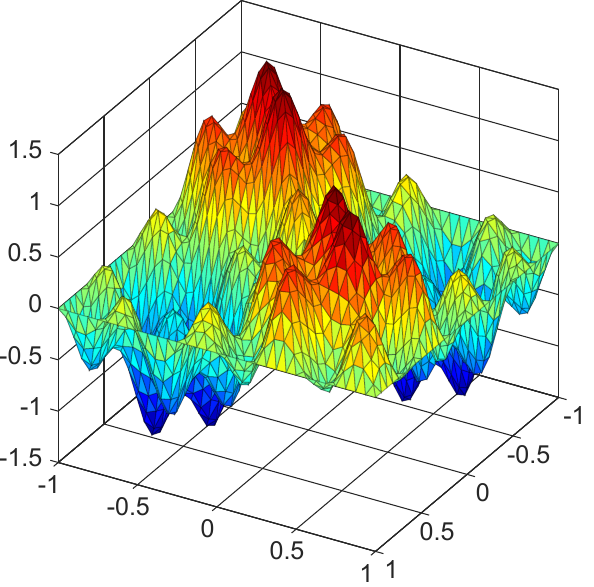}
}
\endgroup
\caption{The comparison of the direct-solved approximation $u_h$ and the multigrid 
iterate $u_h^{(2)}$ in the first test problem.}
\label{fig:ex1-soln}
\end{figure}


The second test problem tests the stopping criterion on a non-uniform mesh 
for the Kellogg intersecting interface problem. The Kellogg problem with a 
checkerboard coefficient 
distribution \cite{Kellogg74interface} is a
commonly used benchmark for testing the efficiency 
and robustness of a posteriori error estimators
(\cite{Cai12eqrobust,Cai15recovery,Cai16robust,Chen02efficiency,Petzoldt02discont}):
\begin{equation}\label{Kellogg}
-\divv(A\nabla u) = 0, \text{ in } \Omega= (-1, 1)^2
\end{equation}
with Dirichlet boundary condition, where the diffusion coefficient $A$ is given by
\[
A= 
\begin{cases} 
R & \text{ in } (0, 1)^2 \cup (-1, 0)^2,
\\[2mm]
1 & \text{ in } \Omega\backslash \Big((0, 1)^2 
\cup (-1, 0)^2\Big).
\end{cases}
\]
The exact solution $u$ of \eqref{Kellogg} is given in polar coordinates 
$(r,\theta)$:
\[
u = r^{\gamma} \psi(\theta) \in H^{1+\gamma-\epsilon}(\Omega)
\,\,\text{ for any }\,\, \epsilon >0,
\]
where the definition of $\psi(\theta) $ is given in, e.g., \cite{Chen02efficiency}.
Here the parameters are:
\[
\gamma = 0.5, \quad R \approx 5.8284271247461907, \quad
\rho = \pi/4, \,\text{ and }\, \sigma \approx -2.3561944901923448.
\]

For this example, $\cT$ is a graded mesh on which the relative error for the 
direct solve $\norm{u - u_{_{\cT}}}_A/ \norm{u}_A \approx 10\%$, in addition, 
we choose $\varepsilon^{-1} = 0.67$ and $\varepsilon_{\rho} = 0.1$ 
for the stopping criterion. The stopping criterion \eqref{alg:stop} is checked 
every three $V(1,1)$-cycles. The local error distribution is shown in Figure 
\ref{fig:ex3-err}.

\begin{table}[h]
\caption{The number of iterations and the total and algebraic errors for the Kellogg problem.} 
\centering  
\begin{tabular}{|c|c|c|c|c|} 
\hline & Stopping  &\# Iter 
& $\norm{u - {u}_{_\cT}^{(k)}}_A$ & $\norm{u_{_\cT} - u_{_\cT}^{(k)}}_A$ 
\\[1mm]
\hline
MG V(1,1)  &  $\eta_a\leq 0.67 \eta_d$ & 2 & $0.05141$ & $1.577\times 10^{-3}$  
\\[1mm]
\hline
MG V(1,1)  &  ${\norm{\mathbf{r}_{k}}_2}/{\norm{\mathbf{r}_0}_2}\leq 10^{-7}$ & 
6 & $0.05139$  &   $8.026\times 10^{-8}$
\\[1mm]
\hline
\end{tabular}
\label{table:ex3-comp}
\end{table}

\begin{figure}[h]
\centering
\subfloat[$\norm{u - u_{_\cT}}_{A,K}$, the local energy error.]
{
\includegraphics[width=0.28\textwidth]{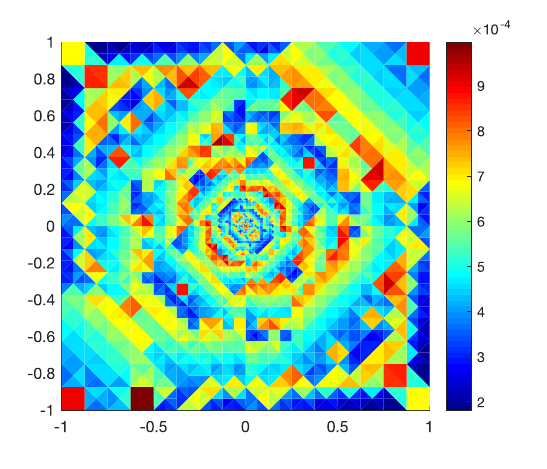}
}
\quad
\subfloat[The local error indicator $\eta_{d,K}$ in \eqref{eq:est-eta} using direct 
solve $u_{_\cT}$.]
{
\includegraphics[width=0.28\textwidth]{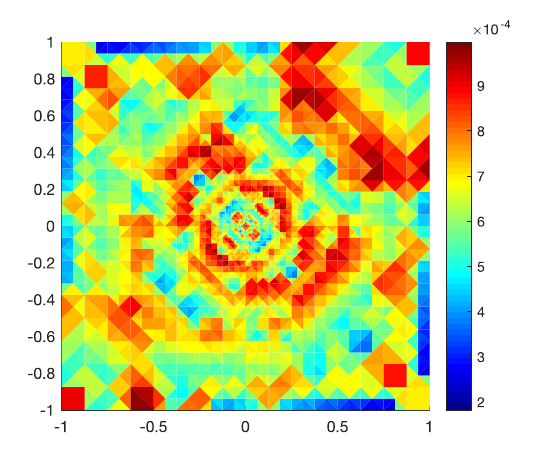}
}
\quad
\subfloat[The local error indicator $\eta_{d,K}$ using iterate $\bar{u}_{_\cT}$.]
{
\includegraphics[width=0.28\textwidth]{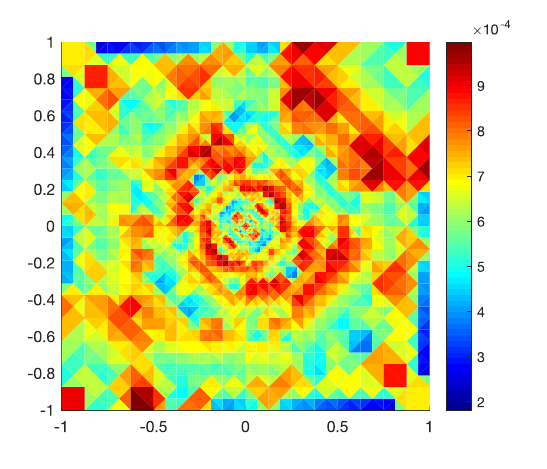}
}
\caption{The comparison the local error and the error indicator distributions.}
\label{fig:ex3-err}
\end{figure}

%

\end{document}